\documentclass[12pts]{article}
\usepackage{amsmath,amsfonts,amssymb,amsthm,mathtools}
\usepackage{microtype}
\usepackage{lipsum}
\usepackage{url,hyperref}
\usepackage{indentfirst}
\hypersetup{colorlinks, linkcolor={red}, citecolor={blue}, urlcolor={black}}
\usepackage{comment}
\usepackage[margin=1.25in]{geometry}
\usepackage{enumerate,enumitem}
\usepackage{subcaption}
\usepackage{tikz}
\usepackage{array,multirow}
\usepackage[capitalise]{cleveref}

\allowdisplaybreaks

\newcommand{\tauo}{\tau_{\text{o}}}
\DeclareMathOperator\inv{inv}
\DeclareMathOperator\des{des}
\DeclareMathOperator\asc{asc}
\DeclareMathOperator\ninv{ninv}

\newcommand*{\myproofname}{Proof}
\newenvironment{claimproof}[1][\myproofname]{\begin{proof}[#1]}{\end{proof}}

\newtheorem{theorem}{Theorem}[section]
\newtheorem{lemma}[theorem]{Lemma}
\newtheorem{example}[theorem]{Example}

\newtheorem{remark}[theorem]{Remark}

\newtheorem{claim}[theorem]{Claim}

\newcommand*{\ceilfrac}[2]{\mathopen{}\left\lceil\frac{#1}{#2}\right\rceil\mathclose{}}
\newcommand*{\floorfrac}[2]{\mathopen{}\left\lfloor\frac{#1}{#2}\right\rfloor\mathclose{}}

\title{Descents and inversions in powers of permutations}
\author{Stijn Cambie \thanks{Department of Computer Science, KU Leuven Campus Kulak-Kortrijk, 8500 Kortrijk, Belgium. Supported by a postdoctoral fellowship by the Research Foundation Flanders (FWO) with grant number 1225224N. Email: {\tt stijn.cambie@hotmail.com}.} \and Jun Yan\thanks{Mathematics Institute, University of Warwick, UK. Email: {\tt jun.yan@warwick.ac.uk}. Supported by the Warwick Mathematics Institute CDT and funding from the UK EPSRC (Grant number: EP/W523793/1).}}
\date{}

\begin{document}

\maketitle

\begin{abstract}
In this paper, we generalise several recent results by Archer and Geary on descents in powers of permutations, and confirm all their conjectures. Specifically, for all $k\in\mathbb{Z}^+$, we prove explicit formulas for the expected numbers of descents and inversions in the $k$-th powers of permutations in $\mathcal{S}_n$ for all $n\geq2k+1$. We also compute the number of Grassmanian permutations in $\mathcal{S}_n$ whose $k$-th powers remain Grassmanian, and the number of permutations in $\mathcal{S}_n$ whose $k$-th powers have the maximum number of descents. 
\end{abstract}

\section{Introduction}
Given a permutation $\pi\in\mathcal{S}_n$, a \textit{descent} in $\pi$ is an index $i\in[n-1]$ satisfying $\pi(i)>\pi(i+1)$, while an \textit{inversion} in $\pi$ is a pair $i<j$ of indices in $[n]$ satisfying $\pi(i)>\pi(j)$. The number of descents and the number of inversions in $\pi$ are denoted by $\des(\pi)$ and $\inv(\pi)$, respectively.

It is easy to show that the expected number of descents in a random permutation $\pi\in\mathcal{S}_n$ is $\frac{n-1}2$, as for each $i\in[n-1]$, the events $\pi(i)>\pi(i+1)$ and $\pi(i)<\pi(i+1)$ are equally likely. A similar argument shows that the expected number of inversions in a permutation $\pi\in\mathcal{S}_n$ is $\frac{n(n-1)}{4}$. 

Recently, in \cite{AG}, while studying the number of permutations $\pi\in\mathcal{S}_n$ whose square, or whose cube, have a fixed small number of descents, Archer and Geary conjectured that for all but the first few values of $n$, the expected number of descents in $\pi^2$ and in $\pi^3$ for $\pi\in\mathcal{S}_n$ are both $\frac{n-1}2-\frac2n$. 

In this paper, we confirm this conjecture. Moreover, we prove explicit formulas for the expected numbers of descents and inversions in the $k$-th powers of permutations in $\mathcal{S}_n$ for all $k\in\mathbb{Z}^+$ and $n\geq2k+1$.

These formulas will be expressed in terms of several divisor functions. Recall that for $k\in\mathbb{Z}^+$, $\tau(k)$ denotes the number of divisors of $k$ and $\sigma(k)=\sum_{d\mid k}d$ denotes the sum of the divisors of $k$. Let $\nu_2(k)$ be the $2$-adic valuation of $k$, i.e., the number of prime factors $2$ in the prime factorization of $k$. Define $\tauo(k)=\tau\left( k/2^{\nu_2(k)}\right)$ to be the number of odd divisors of $k$. We will show that
\begin{theorem}\label{descent}
For $k\in\mathbb{Z}^+$ and $n\geq2k+1$, the expected number of descents in the $k$-th powers of permutations in $\mathcal{S}_n$ is
$$\frac1{n!}\sum_{\pi\in\mathcal{S}_n}\des(\pi^k)=\frac{n-1}2-\frac{\tau^2(k)-\tau(k)-\tauo(k)+\sigma(k)}{2n}.$$
\end{theorem} 

\begin{theorem}\label{inversion}
For $k\in\mathbb{Z}^+$ and $n\geq2k+1$, the expected number of inversions in the $k$-th powers of permutations in $\mathcal{S}_n$ is
$$\frac1{n!}\sum_{\pi\in\mathcal{S}_n}\inv(\pi^k)=\frac{n(n-1)}4-\frac{(\tau(k)-1)n}6-\frac{\tau^2(k)-\tau(k)-\tauo(k)+\sigma(k)}{12}.$$
\end{theorem}

In~\cref{sec:expectation}, we first prove a few lemmas that count for all pairs $(i,j)$ and $(x,y)$, the number of $\pi$ such that $\pi^k$ sends $(i,j)$ to $(x,y)$. These lemmas are then used to prove our main results that determine the expected number of descents (\cref{descent}) and inversions~(\cref{inversion}) in $\pi^k$ over all $\pi \in \mathcal S_n$. Note that setting $k=2$ or $k=3$ in~\cref{descent} yield the same expectation of $\frac{n-1}2-\frac2n$, which confirms~\cite[Conj.~6.1]{AG}.

In Section \ref{sec:reldir}, we consider \textit{Grassmanian permutations}, which are permutations $\pi$ with $\des(\pi) \le 1$. Specifically, we compute the number of Grassmanian permutations whose $k$-th power is also Grassmanian. By~\cite[Lem.~2.2]{AG}, it is sufficient to determine the number of such permutations satisfying $\pi(1) \not=1$ and $\pi(n) \not=n$, as those with $\pi(1)=1$ or $\pi(n)=n$ can be counted recursively. The following result shows that any such permutation $\pi$ is either a cyclic shift or satisfies $\pi^k=\text{id}$ or $\pi^{k-1}=\text{id}$.

\begin{theorem}\label{thr:twiceGrassmanian}
    Let $k\geq3$. If a Grassmanian permutation $\pi \in \mathcal S_n$ satisfies $\pi(1)\not=1$, $\pi(n)\not=n$ and $\des({\pi^k})=1$, then either
    \begin{itemize}
        \item there exists some $s\in[n]$ such that $\pi(i)\equiv i+s\pmod n$ for all $i\in[n]$, or
        \item $\pi$ is a $(k-1)$-th root of the identity permutation, i.e., $\pi^{k-1}=\text{id}$.
    \end{itemize}
\end{theorem}

Since the only Grassmanian with $\pi(1) \not=1, \pi(n) \not=n$ and $\pi^2=\text{id}$ is a cyclic shift permutation (see~\cite[Thm.~2.3]{AG}), the $k=3$ case of Theorem \ref{thr:twiceGrassmanian} confirms~\cite[Conj.~3.2]{AG}. As cyclic shifts are easy to handle, the problem reduces to the following result enumerating Grassmanian permutations that are $k$-th roots of the identity permutation. The case when $k$ is prime gives a nice formula. 

\begin{theorem}\label{thr:Grassmanianqthpower}
For every $k\geq2$, let $\mathcal{D}_k$ be the set of divisors of $k$ excluding 1, and let $$N_k=\frac1k\sum_{d\mid k, d\not=k}\mu(d)(2^{\frac kd}-2).$$
Then the number of Grassmanian permutations $\pi\in\mathcal{S}_n$ with $\pi(1)\not=1$, $\pi(n)\not=n$ and $\pi^k=\text{id}$ is equal to the number of solutions in non-negative integers of the linear equation
\begin{equation}\label{eq:count_total_order}
\sum_{d \in \mathcal D_k}  d \cdot \sum_{i=1}^{N_d} x_{d,i} = n.
\end{equation}
In particular, for a prime $p\geq2$, $N_p=\frac1p(2^p-2)$ and the number of Grassmanian permutation $\pi\in\mathcal{S}_n$ with $\pi(1)\not=1$, $\pi(n)\not=n$ and $\pi^p=\text{id}$ is $\binom{\frac np+N_p-1}{N_p-1}$ if $p\mid n$, and 0 otherwise. 
\end{theorem}

Finally, we provide a short answer to~\cite[Ques.~5.3]{AG}, which asks for the number of permutations whose $k$-th powers have the maximum number of descents, or equivalently are equal to the decreasing permutation. By doing this for any positive integer $k,$ this generalizes~\cite[Thm.~5.1, 5.2]{AG}. 
Let $d_1, d_2, \ldots, d_r$ be the divisors of $k$ with the same $2$-adic valuation as $k$. Define $$S_k(n)=\left\{ (a_1, ..., a_r) \mid a_i \in \mathbb N,\mbox{ } \sum_{i=1}^r a_i d_i = \floorfrac n2 \right\}.$$

\begin{theorem}
    The number of $\pi\in\mathcal S_n$ such that $\pi^k$ is the decreasing permutation is

     $$\sum_{(a_1, ..., a_r) \in S_k(n)} \frac{\floorfrac n2 ! \cdot \prod_{i=1}^r 2^{a_i(d_i-1)} }{\prod _{i=1}^r a_i!d_i^{a_i}  }  $$
\end{theorem}    

\begin{proof}
    Since $\pi^k$ is the decreasing permutation and $\pi^{2k}$ is the identity, the only possible fixed point of $\pi$ is $\ceilfrac{n}{2}$ when $n$ is odd. Also, if a cycle in the cycle decomposition of $\pi$ has length $\ell\geq 2$, then we must have $\ell \mid 2k$ and $\ell \nmid k$. This implies that $\ceilfrac{n}{2}$ is actually a fixed point of $\pi$ when $n$ is odd, and the cycle decomposition of $\pi$ consists of $a_i$ cycles of length $2d_i$ for every $i \in[r]$ for some $(a_1, \ldots, a_r) \in S_k(n)$, and an additional 1-cycle when $n$ is odd. Note that if $j$ is in a cycle of length $2d_i$, then $n-j$ must be in the same cycle at distance $d_i$ away.
    
    Conversely, for each $(a_1, \ldots, a_r) \in S_k(n),$ there are $\frac{\floorfrac n2 !}{\prod_{i=1}^r a_i!(d_i!)^{a_i}} $ many ways to partition the elements into these collection of cycles, and for each cycle there are $(d_i-1)!2^{d_i-1}$ ways to order its elements. Each of these leads to a permutation $\pi$ for which $\pi^k$ is the decreasing permutation.
\end{proof}

    Note that if $\floorfrac n2$ is not a multiple of $2^{\nu_2(k)}$, or equivalently if $n \not \equiv 0,1 \pmod { 2^{\nu_2(k)+1}},$ then $S_k(n)=\emptyset$, so there are no permutations in $\mathcal S_n$ whose $k$-th power is the decreasing permutation.

\section{Expected numbers of descents and inversions}\label{sec:expectation}
Throughout this section, we let $n\geq2k+1\geq3$ and fix distinct $i,j\in[n]$. We first prove a series of lemmas that counts for distinct $x,y\in[n]$, how many $\pi\in\mathcal{S}_n$ satisfies $\pi^k(i)=x$ and $\pi^k(j)=y$. These lemmas will later be used to prove~\cref{descent,inversion}.

\begin{lemma}\label{independent}
If $x,y\in[n]\setminus\{i,j\}$, then the number of permutations $\pi\in\mathcal{S}_n$ satisfying $\pi^k(i)=x$ and $\pi^k(j)=y$ is independent of the choice of $x$ and $y$. 
\end{lemma}
\begin{proof}
This is clear from the symmetry of all elements in $[n]\setminus\{i,j\}$. 
\end{proof}

Though we do not need it to prove Theorem \ref{descent} and Theorem \ref{inversion}, we record here for completeness that the number of such permutations is 
$$(n^2-(2\tau(k)+3)n+\tau^2(k)+3\tau(k)+\tauo(k)+\sigma(k))(n-4)!.$$
This formula can be obtained either as a corollary of the following series of lemmas, or by counting directly as in the proofs of those lemmas.



 \begin{lemma}\label{clm:often_half}
    For every $i\in [n-1]$, among all permutations $\pi \in \mathcal S_n$ for which $\{\pi^k(i),\pi^k(i+1)\} \not=\{i,i+1\},$ half of them satisfy $\pi^k(i)>\pi^k(i+1)$ and half of them satisfy $\pi^k(i)<\pi^k(i+1)$.
\end{lemma}
\begin{proof}
By~\cref{independent}, it suffices to consider those $\pi\in\mathcal{S}_n$ such that exactly one of $\pi^k(i),\pi^k(i+1)$ is equal to $i$ or $i+1$. This can be proved using~\cref{itoi,itoj} below, but we provide a more direct bijective proof here.
    
For every $x \in [n] \setminus\{i,i+1\}$, switching the labels $i$ and $i+1$ gives a bijection between $\pi \in \mathcal S_n$ for which $(\pi^k(i),\pi^k(i+1))=(i,x)$ and those satisfying $(\pi^k(i),\pi^k(i+1))=(x,i+1)$, and and a bijection between those satisfying $ (\pi^k(i),\pi^k(i+1))=(x,i)$ and those with $( \pi^k(i),\pi^k(i+1))=(i+1,x)$.
Since $x>i$ if and only if $i+1<x$ under our assumption on $x$, these two bijections swaps whether $\pi^k$ has an ascent or descent at position $i$, implying that there are equally many of them having each.
\end{proof}

\begin{lemma}\label{itoi}
For every $y\in[n]\setminus\{i,j\}$, the number of permutations $\pi\in\mathcal{S}_n$ satisfying $\pi^k(i)=i$ and $\pi^k(j)=y$ is $$(\tau(k)n-\tau^2(k)-\sigma(k))(n-3)!.$$
\end{lemma}
\begin{proof}
Let $d$ be the length of the cycle that $i$ belongs to in $\pi$, then $d\mid k$ as $\pi^k(i)=i$. Note that $j$ cannot be in the same cycle as $i$, as otherwise $\pi^k(j)=j$. Let $\ell$ be the length of the cycle that $j$ belongs to, and observe that $y$ must be in this cycle as well. Let $1\leq t\leq \min\{k,\ell-1\}$ be the distance from $j$ to $y$ in this cycle, or equivalently the smallest positive integer such that $\pi^t(j)=y$. 

If $t=k$, then $\ell\geq k+1$. On the other hand, for any $\pi\in\mathcal{S}_n$, $d\mid k$ and $k+1\leq\ell\leq n-d$, such that $i$ is in a length $d$ cycle and $j,y$ are in another length $\ell$ cycle with the distance from $j$ to $y$ on this cycle being $k$, we have $\pi^k(i)=i$ and $\pi^k(j)=y$. There are $\sum_{d\mid k}(n-d-k)(n-3)!$ permutations of this form.

If $t<k$, then as $\pi^k(j)=y$, we must have $k\equiv t\pmod\ell$, and so $\ell\leq k$. Moreover, if $\ell\mid k$, then $t=0$, which is not allowed. Conversely, for any $d\mid k$ and every $\ell\in[k]$ not dividing $k$, there is exactly one choice of $t\in[\ell-1]$ satisfying $k\equiv t\pmod\ell$. For any $\pi\in\mathcal{S}_n$ such that $i$ is in a length $d$ cycle and $j,y$ are in another length $\ell$ cycle with the distance from $j$ to $y$ on this cycle being $t$, we have $\pi^k(i)=i$ and $\pi^k(j)=y$. There are $\sum_{d\mid k}\sum_{\ell\nmid k}(n-3)!$ permutations of this form. 

Therefore, the number of permutations $\pi\in\mathcal{S}_n$ satisfying $\pi^k(i)=i$ and $\pi^k(j)=y$ is 
\begin{align*}
\sum_{d\mid k}(n-d-k)(n-3)!+\sum_{d\mid k}\sum_{\ell\nmid k}(n-3)!&=(\tau(k)n-\sigma(k)-k\tau(k)+\tau(k)(k-\tau(k)))(n-3)!\\
&=(\tau(k)n-\tau^2(k)-\sigma(k))(n-3)!,
\end{align*}
as required.
\end{proof}

\begin{lemma}\label{itoj}
For every $y\in[n]\setminus\{i,j\}$, the number of permutations $\pi\in\mathcal{S}_n$ satisfying $\pi^k(i)=j$ and $\pi^k(j)=y$ is $$(n-\tau(k)-\tauo(k))(n-3)!.$$
\end{lemma}
\begin{proof}
From assumption, $i,j,y$ are in the same cycle of $\pi$. Let $\ell$ be the length of this cycle, and let $1\leq t\leq\min\{\ell-1,k\}$ be the distance from $i$ to $j$ on this cycle, which is also the smallest positive integer such that $\pi^t(i)=j$. It follows that $k\equiv t\pmod\ell$, and so $t$ must be the distance from $j$ to $y$ on this cycle as well. In particular, $\ell\not=2t$, as otherwise $i=\pi^\ell(i)=\pi^{2t}(i)=\pi^t(j)=y$, contradiction. 

If $t=k$, then $\ell\geq k+1$ and $\ell\not=2k$. On the other hand, for any $\ell\geq k+1$ and not equal to $2k$, and any permutation $\pi\in\mathcal{S}_n$ with a cycle of length $\ell$ containing $i,j,y$, such that the distance in this cycle from $i$ to $j$ and from $j$ to $y$ are both $k$, we have $\pi^k(i)=j$ and $\pi^k(j)=y$. There are $(n-k-1)(n-3)!$ permutations of this form.

If $t<k$, we must have $\ell<k$ as $k\equiv t\pmod\ell$. If $\ell\mid k$, then $t=0$, which is not allowed. Note that given $k\equiv t\pmod\ell$, then $\ell=2t$, which is also not allowed from above, implies that $2k$ is an odd multiple of $\ell$. Conversely, for any $\ell\in[k]$ not dividing $k$, and such that $2k$ is not an odd multiple of $\ell$, we have that $t\in[\ell-1]$ given by $t\equiv k\pmod\ell$ is not equal to $\frac12\ell$. Moreover, for any permutation $\pi\in\mathcal{S}_n$ with a cycle of length $\ell$ containing $i,j,y$, such that the distance in this cycle from $i$ to $j$ and from $j$ to $y$ are both $t$, we have $\pi^k(i)=j$ and $\pi^k(j)=y$. Let $k=2^ab$, where $a,b$ are positive integers and $2\nmid b$. If $\ell\in[k]$, $\ell\nmid k$ and $2k$ is an odd multiple of $\ell$, then $\ell$ must be of the form $2^{a+1}d$, where $d$ is a proper divisor of $b$. As the converse of this is also true, the number of such $\ell$ is then equal to $\tau(b)-1=\tauo(k)-1$. Hence, the number of the permutations of this form is $(k-\tau(k)-\tauo(k)+1)(n-3)!$.

Therefore, the number of permutations $\pi\in\mathcal{S}_n$ satisfying $\pi^k(i)=j$ and $\pi^k(j)=y$ is
$$((n-k-1)+(k-\tau(k)-\tauo(k)+1))(n-3)!=(n-\tau(k)-\tauo(k))(n-3)!,$$
as required.
\end{proof}

\begin{lemma}\label{ijfixed}
The number of permutations $\pi\in\mathcal{S}_n$ satisfying $\pi^k(i)=i$ and $\pi^k(j)=j$ is $$(\tau^2(k)-\tau(k)+\sigma(k))(n-2)!.$$
\end{lemma}
\begin{proof}
First suppose $i$ and $j$ are in distinct cycles of $\pi$ of length $d_1$ and $d_2$, respectively. From assumption, we must have $d_1,d_2\mid k$. On the other hand, for any $d_1,d_2\mid k$ and any $\pi\in\mathcal{S}_n$ with $i$ in a cycle of length $d_1$, and $j$ in another cycle of length $d_2$, we have $\pi^k(i)=i$ and $\pi^k(j)=j$. There are $\tau^2(k)(n-2)!$ permutations of this form. 

Now suppose $i$ and $j$ are in the same cycle of $\pi$ of length $d$. Again, $d\mid k$ from assumption. Conversely, for any $d\mid k$, and any $\pi\in\mathcal{S}_n$ with $i,j$ in the same cycle of length $d$, we have $\pi^k(i)=i$ and $\pi^k(j)=j$. Since the distance from $i$ to $j$ in this cycle can be any number in $[d-1]$, there are $\sum_{d\mid k}(d-1)(n-2)!=(\sigma(k)-\tau(k))(n-2)!$ permutations of this form. 

Therefore, the number of permutations $\pi\in\mathcal{S}_n$ satisfying $\pi^k(i)=i$ and $\pi^k(j)=j$ is $$\tau^2(k)(n-2)+(\sigma(k)-\tau(k))(n-2)!=(\tau^2(k)-\tau(k)+\sigma(k))(n-2)!,$$ as required.
\end{proof}

\begin{lemma}\label{ijexchange}
The number of permutations $\pi\in\mathcal{S}_n$ satisfying $\pi^k(i)=j$ and $\pi^k(j)=i$ is $$\tauo(k)(n-2)!.$$
\end{lemma}
\begin{proof}
From assumption, $i$ and $j$ are in the same cycle in $\pi$. Let $\ell$ be the length of this cycle. Let $1\leq t\leq\min\{k,\ell-1\}$ be minimal so that $\pi^t(i)=j$, or equivalently $t$ is the distance from $i$ to $j$ in the cycle, then $t\equiv k\pmod\ell$. It also follows that $\pi^{\ell-t}(j)=i$, so $\ell-t\equiv -t\equiv k\pmod\ell$. Thus, $t\equiv -t\pmod\ell$, and since $t\in[\ell-1]$, we must have $\ell=2t$. Since $\pi^{2k}(i)=\pi^k(j)=i$, we must have $\ell\mid2k$ and so $t\mid k$. Also, $2t\nmid k$ as otherwise $\ell\mid k$ and $\pi^{k}(i)=i$.

Conversely, for any $t\geq1$ satisfying $t\mid k$ and $2t\nmid k$, we have $k\equiv t\pmod {2t}$. So for any $\pi\in\mathcal{S}_n$ containing a cycle of length $2t$, in which $i$ and $j$ are distance $t$ apart, we have $\pi^k(i)=j$ and $\pi^k(j)=i$. 

Therefore, the number of such permutations in $\mathcal{S}_n$ is exactly $(n-2)!$ times the number of divisors $t$ of $k$ such that $2t\nmid k$. Let $k=2^ab$, where $a,b$ are positive integers and $2\nmid b$. Then, every divisor $t$ of $k$ satisfying $2t\nmid k$ is of the form $2^ad$, where $d$ is a divisor of $b$, and the converse is true as well. Thus, there are exactly $\tau(b)=\tauo(k)$ such divisors, which proves that there are exactly $\tauo(k)(n-2)!$ permutations $\pi\in\mathcal{S}_n$ satisfying $\pi^k(i)=j$ and $\pi^k(j)=i$. 
\end{proof}

We now combine these lemmas to prove~\cref{descent,inversion}. We prove~\cref{inversion} first as its proof contains most of what we need to prove~\cref{descent}.

\begin{proof}[Proof of Theorem \ref{inversion}]
For each $\pi\in\mathcal{S}_n$, let $\ninv(\pi)=\frac{n(n-1)}{2}-\inv(\pi)$ be the number of pairs of indices $i<j$ in $[n]$ satisfying $\pi(i)<\pi(j)$. We call each pair $i<j$ of this form a \textit{non-inversion} of $\pi$. Note that to prove Theorem \ref{inversion}, it suffices to show that 
$$\sum_{\pi\in\mathcal{S}_n}(\ninv(\pi^k)-\inv(\pi^k))=\frac{(\tau(k)-1)n\cdot n!}3+\frac{(\tau^2(k)-\tau(k)-\tauo(k)+\sigma(k))n!}{6}.$$

Let $\pi\in\mathcal{S}_n$ and $i<j$ in $[n-1]$ satisfy $\pi^k(i)=x$ and $\pi^k(j)=y$. We consider all possible types of ways we can have an inversion or non-inversion on the pair $i<j$, depending on the values of $x$ and $y$.

\textbf{Type 1.} $x,y\not\in\{i,j\}$. By Lemma \ref{independent}, it is equally likely to have $x<y$ or $x>y$, so inversions and non-inversions of this type cancel out.

\textbf{Type 2.} $x=i$ and $y\not\in\{i,j\}$. By Lemma \ref{itoi}, there are $(i-1)(\tau(k)n-\tau^2(k)-\sigma(k))(n-3)!$ ways to have an inversion of this type as $y$ can take any value in $[i-1]$, and $(n-i-1)(\tau(k)n-\tau^2(k)-\sigma(k))(n-3)!$ ways for a non-inversion as $y$ can take any value in $[n]\setminus([i]\cup\{j\})$.

\textbf{Type 3.} $x=j$ and $y\not\in\{i,j\}$. By Lemma \ref{itoj}, there are $(j-2)(n-\tau(k)-\tauo(k))(n-3)!$ ways to have an inversion of this type, and $(n-j)(n-\tau(k)-\tauo(k))(n-3)!$ ways for a non-inversion.

\textbf{Type 4.} $x\not\in\{i,j\}$ and $y=j$. By Lemma \ref{itoi}, there are $(n-j)(\tau(k)n-\tau^2(k)-\sigma(k))(n-3)!$ ways to have an inversion of this type, and $(j-2)(\tau(k)n-\tau^2(k)-\sigma(k))(n-3)!$ ways for a non-inversion. 

\textbf{Type 5.} $x\not\in\{i,j\}$ and $y=i$. By Lemma \ref{itoj}, there are $(n-i-1)(n-\tau(k)-\tauo(k))(n-3)!$ ways to have an inversion of this type, and $(i-1)(n-\tau(k)-\tauo(k))(n-3)!$ ways for a non-inversion.

\textbf{Type 6.} $x=i,y=j$. By Lemma \ref{ijfixed}, there are $(\tau^2(k)-\tau(k)+\sigma(k))(n-2)!$ ways for this to happen, each resulting in a non-inversion.

\textbf{Type 7.} $x=j,y=i$. By Lemma \ref{ijexchange}, there are $\tauo(k)(n-2)!$ ways for this to happen, each resulting in an inversion.    

Summing from Type 2 to Type 7, the total number of ways to have a non-inversion at $i<j$ of these types is
$$((n-i+j-3)(\tau(k)n-\tau^2(k)-\sigma(k))+(n+i-j-1)(n-\tau(k)-\tauo(k)))(n-3)!+(\tau^2(k)-\tau(k)+\sigma(k))(n-2)!,$$
while the total number of ways to have an inversion at $i<j$ of these types is
$$((n+i-j-1)(\tau(k)n-\tau^2(k)-\sigma(k))+(n-i+j-3)(n-\tau(k)-\tauo(k)))(n-3)!+\tauo(k)(n-2)!.$$
The difference of the first two terms of the two expressions above, summed over all $i<j$ in $[n]$, is
\begin{align*}
&\phantom{==}\sum_{i=1}^{n-1}\sum_{j=i+1}^n2(j-i-1)((\tau(k)n-\tau^2(k)-\sigma(k))-(n-\tau(k)-\tauo(k)))(n-3)!\\
&=((\tau(k)-1)n-\tau^2(k)-\sigma(k)+\tau(k)+\tauo(k))(n-3)!\sum_{i=1}^{n-1}\sum_{j=i+1}^n2(j-i-1)\\
&=((\tau(k)-1)n-\tau^2(k)-\sigma(k)+\tau(k)+\tauo(k))(n-3)!\frac{n(n-1)(n-2)}3\\
&=((\tau(k)-1)n-\tau^2(k)-\sigma(k)+\tau(k)+\tauo(k))\frac{n!}3.
\end{align*}
The difference of the last term of the two expressions above, again summed over all $i<j$ in $[n]$, is
$$(\tau^2(k)-\tau(k)+\sigma(k)-\tauo(k))(n-2)!\frac{n(n-1)}{2}=(\tau^2(k)-\tau(k)+\sigma(k)-\tauo(k))\frac{n!}{2}.$$
Hence,
\begin{align*}
&\phantom{=}\sum_{\pi\in\mathcal{S}_n}(\ninv(\pi^k)-\inv(\pi^k))\\
&=((\tau(k)-1)n-\tau^2(k)-\sigma(k)+\tau(k)+\tauo(k))\frac{n!}3+(\tau^2(k)-\tau(k)+\sigma(k)-\tauo(k))\frac{n!}{2}\\
&=\frac{(\tau(k)-1)n\cdot n!}3+\frac{(\tau^2(k)-\tau(k)-\tauo(k)+\sigma(k))n!}{6},
\end{align*}
as required.
\end{proof}

\begin{proof}[Proof of Theorem \ref{descent}]
For each $\pi\in\mathcal{S}_n$, let $\asc(\pi)=n-1-\des(\pi)$ be the number of indices $i\in[n-1]$ satisfying $\pi(i)<\pi(i+1)$. We say that $\pi$ has an \textit{ascent} at $i$ in these situations. Note that to prove Theorem \ref{descent}, it suffices to show that 
$$\sum_{\pi\in\mathcal{S}_n}(\asc(\pi^k)-\des(\pi^k))=(\tau^2(k)-\tau(k)-\tauo(k)+\sigma(k))(n-1)!.$$

Let $\pi\in\mathcal{S}_n$ and $i\in[n-1]$ satisfy $\pi^k(i)=x$ and $\pi^k(i+1)=y$. Setting $j=i+1$ in Type 1 to Type 7 in the proof of Theorem \ref{inversion} above, we see that if $x,y\not\in\{i,i+1\}$, descents and ascents of Type 1 cancel out, and  by summing from Type 2 to Type 7, the total number of ways to have an ascent at $i$ of those types is
$$(\tau(k)n-\tau^2(k)-\sigma(k))(n-2)!+(n-\tau(k)-\tauo(k))(n-2)!+(\tau^2(k)-\tau(k)+\sigma(k))(n-2)!,$$
while the total number of ways to have an descent at $i$ of those types is
$$(\tau(k)n-\tau^2(k)-\sigma(k))(n-2)!+(n-\tau(k)-\tauo(k))(n-2)!+\tauo(k)(n-2)!.$$
Hence, 
\begin{align*}
\sum_{\pi\in\mathcal{S}_n}(\asc(\pi^k)-\des(\pi^k))&=\sum_{i=1}^{n-1}(\tau^2(k)-\tau(k)+\sigma(k)-\tauo(k))(n-2)!\\
&=(\tau^2(k)-\tau(k)-\tauo(k)+\sigma(k))(n-1)!,
\end{align*}
as required.
\end{proof}

Note that as both $\tau^2(k)-\tau(k)$ and $\tauo(k)-\sigma(k)$ are even, $\frac12(\tau^2(k)-\tau(k)-\tauo(k)+\sigma(k))$ is an integer. In the case when $k=p$ is an odd prime, the formula for the expectation of $\des(\pi^p)$ simplifies to $\frac{n-1}{2}-\frac{p+1}{2n}$, while the one for $\inv(\pi^p)$ simplifies to $\frac{n(n-1)}{4}-\frac n6-\frac{p+1}{12}$.

Finally, we remark that \cref{descent} is actually valid for every $n\ge k+\ell(k)$, where $\ell(k)$ is defined to be the largest proper divisor of $k$. By~\cref{clm:often_half}, we only need to compare permutations counted in~\cref{ijfixed,ijexchange} when $j=i+1$.
There are only two situations there in which we considered a union of cycles with total length at least $k+\ell(k)$: the disjoint union of two cycles of length $k$, one containing $i$ and the other containing $i+1$ in~\cref{ijfixed}, and a cycle of length $2k$ with $i$ and $i+1$ being diametrically opposite in~\cref{ijexchange}. The former contributes $(n-2)!$ ascents while the latter $(n-2)!$ descents, which cancel out. Thus, the formula is valid for every $n\ge k+\ell(k).$



\section{On Grassmanian permutations $\pi$ with $\des(\pi^k) \in \{0,1\}$}\label{sec:reldir}

In this section, we show that a permutation being Grassmanian is so rare, that both $\pi$ and $\pi^k$ being so implies a lot of structure, and so we can count the number of them precisely. By~\cite[Lem.~2.2]{AG}, it is sufficient to determine the number of such permutations for which $\pi(1) \not=1$ or $\pi(n) \not=n$, as those satisfying $\pi(1)=1$ or $\pi(n)=n$ can be counted recursively.

We begin with a technical lemma that we will use shortly to prove Theorem \ref{thr:twiceGrassmanian}.
\begin{lemma}\label{lemma:k-1isid}
Let $k\geq3$, $i,j\in[n-1]$. Suppose $\pi,\pi^k$ are both Grassmanian permutations with $\pi(i)=\pi^k(j)=n$, $\pi(i+1)=\pi^k(j+1)=1$. If there exists some $0\leq t\leq i-1$, such that $\pi(i-\ell)=\pi^k(j-\ell)=n-\ell$ for all $\ell\in[t]$, and $\pi(n)=\pi^k(n)=n-t-1$, then $i=j$ and $\pi^{k-1}=\text{id}$.
\end{lemma}
\begin{proof}
Since $\pi(n)=\pi^k(n)=n-t-1$, we have $\pi^{k-1}(n)=n$. Thus, $\pi^{k-2}(n)=i$ as $\pi(i)=n$, and so $\pi^{k-1}(i)=i$. However, we also have $\pi^{k-1}(j)=i$ as $\pi^k(j)=n$. Hence, we must have $i=j$. 

We now use induction to show that $\pi^{k-1}(n-\ell)=n-\ell$ for all $0\leq\ell\leq n-1$, or equivalently $\pi^{-1}(n-\ell)=\pi^{-k}(n-\ell)$. The case $\ell=0$ follows from above. For all $\ell\in[t]$, since $i=j$, from assumption we have $\pi(i-\ell)=\pi^k(i-\ell)=n-\ell$, and so $\pi^{k-1}(n-\ell)=n-\ell$. The case $\ell=t+1$ follows similarly from $\pi(n)=\pi^k(n)=n-t-1$. Now assume $r>t+1$ and $\pi^{k-1}(n-\ell)=n-\ell$ for all $0\leq\ell\leq r-1$. 

Let $x,y$ be such that $\pi(x)=\pi^k(y)=n-r$. From induction hypothesis, $\pi^{-1}(n-\ell)=\pi^{-k}(n-\ell)$ for all $0\leq\ell\leq r-1$. Let $x_1$ be the smallest index such that $\pi(x_1)\in\{n-r+1,\ldots,n\}$ and let $x_2$ be the smallest index larger than $i+1$ such that $\pi(x_2)\in\{n-r+1,\ldots,n\}$. Since the unique descent of $\pi$ is at $i$, we must have $x,y\in\{x_1-1,x_2-1\}$. If $x=y$, then $\pi^{k-1}(n-r)=n-r$ and we are done. Otherwise, $x=x_2-1$ or $y=x_2-1$. Since $x_2>i+1$, we have $\pi(x_2)<\cdots<\pi(n)=n-t-1$, which implies $n-r+1\leq\pi(x_2)\leq x_2-t-1$, and so $x_2-1\geq n-r+1+t\geq n-r+1$. If $x=x_2-1$, then from induction hypothesis, $\pi^{k-1}(x)=\pi^{k-1}(x_2-1)=x_2-1=x$, so $n-r=\pi(x)=\pi^k(x)$ and $\pi^{k-1}(n-r)=n-r$, as required. If $y=x_2-1$, then similarly, $n-r=\pi^k(y)=\pi(\pi^{k-1}(y))=\pi(y)$ and so $\pi^{k-1}(n-r)=n-r$ as well. This completes the induction and the proof that $\pi^{k-1}=\text{id}$.
\end{proof}

\begin{proof}[Proof of~\cref{thr:twiceGrassmanian}]
It is clear that a Grassmanian permutation $\pi\in\mathcal{S}_n$ satisfying $\pi(1)\not=1$, $\pi(n)\not=n$ is of the form 
$$\pi=\pi(1)\ldots\pi(i-1)n1\pi(i+2)\ldots\pi(n),$$
where $i\in[n-1]$, $\pi(1)<\cdots<\pi(i-1)$, and $\pi(i+2)<\cdots<\pi(n)$. Let $\pi$ be such a permutation satisfying $\des({\pi^k})=1$. From above, there exists $i\in[n-1]$ such that $\pi(i)=n$ and $\pi(i+1)=1$. 

\begin{claim}\label{claim:pi^knot1n}
$\pi^k(1)\not=1$ and $\pi^k(n)\not=n$.
\end{claim} 
\begin{claimproof}
Assume for a contradiction that at least one of $\pi^k(1)=1$ and $\pi^k(n)=n$ holds. We first prove that in fact both have to hold simultaneously. Indeed, if $\pi^k(1)=1$ and $\pi^k(j)=n$ for some $j\not=n$, then the unique descent of $\pi^k$ is at $j$. Note that $\pi^{k-1}(1)=i+1$ as $\pi(i+1)=1$. It follows that $\pi^k(i+1)=i+1$. If $j\geq i+1$, then we get $\pi^k(i)=i$ as $1=\pi^k(1)<\cdots<\pi^k(i+1)=i+1$. But this together with $\pi(i)=n=\pi^k(j)$ implies that $\pi^{k-1}(n)=i=\pi^{k-1}(j)$, so $j=n$, contradiction. If $j\leq i$, then from $i+1=\pi^k(i+1)<\cdots<\pi^k(n)\leq n$, we get $\pi^k(n)=n$, so again $j=n$, contradiction. The case when $\pi^k(n)=n$ and $\pi^k(j)=1$ for some $j\not=1$ is similar, so we have both $\pi^k(1)=1$ and $\pi^k(n)=n$. It follows from $\pi(i)=n$ and $\pi(i+1)=1$ that $\pi^k(i)=i$ and $\pi^k(i+1)=i+1$.

Suppose $\{\pi^k(2),\ldots,\pi^k(i-1)\}\not=\{2,\ldots,i-1\}$, then let $j_1\in\{2,\ldots,i-1\}$ be minimal such that $\pi^{-k}(j_1)>i+1$ and let $j_2\in\{i+2,\ldots,n-1\}$ be maximal such that $\pi^{-k}(j_2)<i$. Then $\pi^k$ has descents at both $\pi^{-k}(j_1)-1$ and $\pi^{-k}(j_2)$, contradicting $\des(\pi^k)=1$. Thus $\{\pi^k(2),\ldots,\pi^k(i-1)\}=\{2,\ldots,i-1\}$ and $\{\pi^k(i+2),\ldots,\pi^k(n-1)\}=\{i+2,\ldots,n-1\}$.

We now show that $\pi^k=\text{id}$, which contradicts $\des(\pi^k)=1$ and proves the claim. First, we use induction to show that $\pi^k(\ell)=\ell$ for all $1\leq\ell\leq i-1$. The base case $\ell=1$ follows from assumption. Suppose this is true for all $1\leq\ell\leq t<i-1$, and suppose for a contradiction that $\pi^k(j)=t+1$ for some $j>t+1$. Note that $j<i-1$ and the unique descent of $\pi^k$ must be at $j-1$. It follows that $\pi^k(\ell)=\ell$ for all $\ell\geq i$. Let $x$ be such that $\pi(x)=t+1$, then $\pi^{k-1}(j)=x$. Since $\pi$ has exactly one descent which is at $i$, and $\pi(1)\geq2$, we must have $x\leq t$ or $x\geq i+2$. If $x\leq t$, then from induction hypothesis, we have $\pi^k(x)=x$. But then $\pi^{k-1}(t+1)=x$ as well, so $j=t+1$, contradiction. If $x\geq i+2$, then $\pi^k(x)=x$ from above, so again $\pi^{k-1}(t+1)=x$ and $j=t+1$, contradiction. Similarly, we can use induction to show that $\pi^k(\ell)=\ell$ for all $i+1\leq\ell\leq n$, so $\pi^k=\text{id}$, as required.  
\end{claimproof}

From~\cref{claim:pi^knot1n}, there exists $j\in[n-1]$ such that $\pi^k(j)=n$ and $\pi^k(j+1)=1$. Suppose that $\pi^{k-1}\not=\text{id}$, we show that $\pi$ is a cyclic shift permutation.

First assume that $i\leq j$. We use induction on $\ell$ to show that $\pi(i-\ell)=\pi^k(j-\ell)=n-\ell$ for all $0\leq \ell\leq i-1$. Assuming this, we have $\{\pi(i+2),\ldots,\pi(n)\}=\{2,\ldots,n-i\}$. It then follows from $\pi(i+2)<\cdots<\pi(n)$ that $\pi(i+t)=t$ for $2\leq t\leq n-i$, and thus that $\pi$ is the cyclic shift permutation given by $\pi(\ell)\equiv\ell+n-i\pmod{n}$, as required.

The base case $\ell=0$ follows from assumptions. Now assume this has been proved for all $0\leq\ell\leq t<i-1$, and assume for a contradiction that $\pi(i-t-1)=\pi^k(j-t-1)=n-t-1$ is not true. Note that as $\pi$ and $\pi^k$ both have exactly one descent, we must have either $\pi(i-t-1)=n-t-1$ or $\pi(n)=n-t-1$, and either $\pi^k(j-t-1)=n-t-1$ or $\pi^k(n)=n-t-1$.

\textbf{Case 1.} $\pi(n)=n-t-1$ and $\pi^k(n)=n-t-1$. Then by Lemma \ref{lemma:k-1isid}, $\pi^{k-1}=\text{id}$, contradiction. 

\textbf{Case 2.} $\pi(n)=n-t-1$ and $\pi^k(j-t-1)=n-t-1$. It follows that $\pi^{k-1}(j-t-1)=n$, and so $\pi(j)=j-t-1$ as $\pi^k(j)=n$. Since $j-t-1\not=n$ and $\pi(i)=n$, we have $j\not=i$ and so $j\geq i+1$. It follows that $j-t-1=\pi(j)<\cdots<\pi(n)=n-t-1$, so we must have $\pi(j+\ell)=j-t-1+\ell$ for all $0\leq\ell\leq n-j$. In particular, $\pi(j+1)=j-t$. Since $\pi^k(j+1)=1$, we have $\pi^{k-1}(j-t)=1$ and so $\pi^k(j-t)=\pi(1)$. But from induction hypothesis, $\pi^k(j-t)=n-t=\pi(i-t)$, so $\pi(1)=n-t=\pi(i-t)$, and thus $i-t=1$, contradiction.

\textbf{Case 3.} $\pi(i-t-1)=n-t-1$ and $\pi^k(n)=n-t-1$. Then, $\pi^{k-1}(n)=i-t-1$, and from $\pi(i)=n$ we get $\pi^k(i)=i-t-1$. But as $i\leq j$, we have $1\leq\pi^k(1)<\cdots<\pi^k(i)=i-t-1<i$, contradiction. 

If $i>j$, we can similarly use induction to show that $\pi(i+\ell)=\pi^k(j+\ell)=\ell$ for all $1\leq\ell\leq n-i$, which again implies that $\pi$ is the cyclic shift given by $\pi(\ell)\equiv\ell+n-i\pmod{n}$, as required. \qedhere




\end{proof}

By Theorem \ref{thr:twiceGrassmanian}, and as cyclic shifts are easy to handle, it suffices now to prove Theorem \ref{thr:Grassmanianqthpower}, which counts the number of $k$-th roots of the identity permutation that are Grassmanian. This generalises the essence of~\cite[Thm.~2.3, 3.1]{AG}, where the $k=2,3$ cases are solved. 

We call a permutation whose cycle decomposition is a single cycle of length $n$ an $n$-cycle. Note that an $n$-cycle $\pi$, where $n>1$, automatically satisfies $\pi(1) \not=1$ and $\pi(n) \not=n.$ We first prove a series of lemmas about Grassmanian $n$-cycles.

\begin{lemma}~\cite[Thm.~9.4]{GR93}\label{lem:descentati}
For all $1\leq i\leq n-1$, the number of $n$-cycles with a unique descent at position $i$ is
$$\frac1n\sum_{d\mid \gcd(i,n)}\mu(d)\binom{n/d}{i/d}.$$
\end{lemma}

\begin{lemma}\label{lem:Np}
    The number $N_n$ of Grassmanian $n$-cycles is
    $$\frac1n\sum_{d\mid n,d\not=n}\mu(d)(2^{\frac nd}-2).$$
    In particular, if $n=p$ is prime, then $N_p=\frac1p(2^p-2)$.
\end{lemma}

\begin{proof}
    Every Grassmanian $n$-cycle has a unique descent at some index $i\in[n-1]$, so by Lemma \ref{lem:descentati}, 
    $$N_n=\frac1n\sum_{i=1}^{n-1}\sum_{d\mid\gcd(i,n)}\mu(d)\binom{n/d}{i/d}=\frac1n\sum_{d\mid n,d\not=n}\mu(d)\sum_{i=1}^{\frac nd-1}\binom{n/d}{i}=\frac1n\sum_{d\mid n,d\not=n}\mu(d)(2^{\frac nd}-2).$$
When $n=p$ is a prime number, the sum above contains only one term and equals $\frac1p(2^p-2)$.
\end{proof}

\begin{lemma}\label{lem:existence_composition}
     Let $\alpha\in\mathcal{S}_r, \beta\in\mathcal{S}_s$ be Grassmanian permutations with no fixed point. Then, there exists a Grassmanian permutation $\pi\in\mathcal{S}_{r+s}$ with a partition $[r+s]=A\cup B$, such that the restrictions of $\pi$ to $A$ and $B$ are permutations isomorphic to $\alpha$ and $\beta$, respectively.
\end{lemma}

\begin{proof}
    To distinguish it from $\alpha$, we view $\beta$ as a permutation on the set $[\overline{s}]=\{\overline 1,\overline 2,\ldots,\overline s\}$. For notational convenience, define $f:[r]\cup[\overline{s}]\to[r]\cup[\overline{s}]$ by $f(i)=\alpha(i)$ for $i\in[r]$ and $f(\overline{i})=\beta(\overline{i})$ for $\overline{i}\in[\overline{s}]$. Suppose $\alpha(t)=r, \alpha(t+1)=1$ and $\beta(\overline m)=\overline s, \beta(\overline{m+1})=\overline 1$.
    
    After relabelling each $x_i$ to $i$, the desired permutation $\pi$ is equivalent to an ordering $x_1\prec x_2\prec\cdots\prec x_{r+s}$ of the elements in $[r]\cup[\overline{s}]$, such that the elements in $[r]$ and the elements in $[\overline{s}]$ are still ordered in the usual way, and the sequence $f(x_1), f(x_2),\ldots,f(x_{r+s})$ has exactly one descent in this ordering $\prec$. Note that as $t\prec t+1$ and $f(t+1)=\alpha(t+1)\prec\alpha(t)=f(t)$, there must be a descent somewhere between $t$ and $t+1$. Similarly, there is a descent between $\overline{m}$ and $\overline{m+1}$. Hence, for $f(x_1), f(x_2),\ldots,f(x_{r+s})$ to have at most one descent, we must have $t\prec\overline{m+1}$ and $\overline{m}\prec t+1$, and that the unique descent is at $x_{t+m}$. It follows that the smallest $t+m$ elements under $\prec$ must be $[t]\cup[\overline{m}]$, and we call them the first part of $\prec$. The largest $r+s-t-m$ elements under $\prec$ are called the second part of $\prec$.
    
    We construct such an ordering with the following process. Start with the ordering $1\prec2\prec\cdots\prec t\prec\overline 1\prec\overline 2\prec\cdots\prec\overline{m}\prec t+1\prec\cdots\prec r\prec\overline{m+1}\prec\cdots\prec\overline s$. In every step, if within the same part of $\prec$ there is some $i$ immediately preceding some $\overline{j}$, but $f(i)\succ f(\overline{j})$, then we swap the order of $i$ and $\overline{j}$ in $\prec$. Since elements in $[r]$ are only ever moved up in the ordering $\prec$, if at some point during the process we have $i\succ\overline j$, then this is true from then on. Also, if within the same part we have $i\succ\overline j$ at some point, then $i$ must have been swapped with $\overline j$ in some previous step, so from that point on we have $f(i)\succ f(\overline j)$.
    
    This process must terminate. If it ends with $\overline 1\prec\overline 2\prec\cdots\prec\overline{m}\prec1\prec2\prec\cdots\prec t\prec\overline{m+1}\prec\cdots\prec\overline s\prec t+1\prec\cdots\prec r$, then from the observation above $f(x_1), f(x_2),\ldots,f(x_{r+s})$ has exactly one descent. If the process ends earlier because no further swap is needed, then from the definition of swaps, $f$ is increasing within both parts, and thus has exactly one descent.
\end{proof}

\begin{example}
    Let $\alpha=231$ and $\beta=\overline2\overline5\overline1\overline3\overline4$. The process above starts with $2\prec3\prec\overline2\prec\overline5\prec1\prec\overline1\prec\overline3\prec\overline4$. The first step swaps $3$ and $\overline{2}$, and changes it to $2\prec\overline2\prec3\prec\overline5\prec1\prec\overline1\prec\overline3\prec\overline4 $, because $\alpha(3)=1\succ\overline 5=\beta(\overline 2)$. The process now terminates as no more swap is needed. By relabelling, we get the desired permutation $\pi=34581267$, which indeed has exactly one descent. 
\end{example}

\begin{lemma}\label{lem:determined_order}
    Given distinct Grassmanian cycles $\alpha, \beta$ of length at least two, there is at most one Grassmanian permutation which decomposes exactly into two cycles isomorphic to $\alpha$ and $\beta$.
\end{lemma}

\begin{proof}
    We use the same notations as in the proof of Lemma \ref{lem:existence_composition}. Suppose $\prec_1,\prec_2$ are two different orderings that would produce two distinct Grassmanian permutations satisfying the required conditions. As proved before, $\prec_1,\prec_2$ have the same first part and the same second part.
    
    Suppose $i\in[r]$ is the index that maximises $v_2-v_1$, where $v_1$ is the number of elements in $[\overline{s}]$ that $i$ is greater than under $\prec_1$, or equivalently the unique element such that $\overline{v_1} \prec_1 i \prec_1 \overline{v_1+1}$, and similarly for $v_2$. Let $j$ be this maximum, and note we can assume without loss of generality that $j\ge1$, which means that $i$ surpasses $j$ additional elements of $[\overline{s}]$ when we go from $\prec_1$ to $\prec_2$. Since $i$ is in the same part of $\prec_1$ and $\prec_2$ and $f$ is increasing in that part, it follows that $\alpha(i)$ also surpasses at least $j$ elements of $[\overline{s}]$, and thus exactly $j$ by maximality. Repeating the argument, since $\alpha$ is a cycle, $\alpha^k(i)$ attains all values of $\alpha$, so each $i\in[r]$ surpasses exactly $j$ elements of $[\overline{s}]$ going from $\prec_1$ to $\prec_2$. By symmetry, there is some $j'\geq1$ such that each $\overline{i}\in[\overline{s}]$ is surpassed by $j'$ elements of $[r]$.
    
    It follows that $\prec_1,\prec_2$ both consists of groups of $j+j'$ elements, where in $\prec_1$ each group begins with a block of $j'$ elements in $[r]$ and ends with a block of $j$ elements in $[\overline{s}]$, and in $\prec_2$ the two blocks in each group of $\prec_1$ are swapped. Since these are the only order swaps, $\alpha$ must send each block of $j'$ elements in $[r]$ to another block of $j'$ elements in $[r]$. As $\alpha$ is increasing in each block, it follows that if $i$ is the $\ell$-th element in a block, then so is $\alpha(i)$. Repeating this argument shows that for every $\ell\in[j']$, $\alpha$ only sends the $\ell$-th elements in these block to each other, contradicting $\alpha$ is a cycle if $j'>1$. Similarly, $j>1$ would contradict that $\beta$ is a cycle. Therefore, we must have $j=j'=1$, and so $r=s$. We show that $\alpha$ is isomorphic to $\beta$, which is a contradiction. Indeed, for every $i\in[r]$, the groups containing $i$ is just $i$ and $\overline{i}$. Going from $\prec_1$ to $\prec_2$, $i$ swapped with $\overline{i}$, so $\alpha(i)$ swapped with $\beta(\overline{i})$, implying that they are in the same group and hence $\overline{\alpha(i)}=\beta(\overline{i})$.
\end{proof}

\begin{example}
    A situation where $j=1,j'=2$ could be 
     $3\prec_14\prec_1\overline{2}\prec_15\prec_16\prec_1\overline{3}\prec_11\prec_12\prec_1\overline{1}$ and $\overline{2}\prec_23\prec_24\prec_2\overline{3}\prec_25\prec_26\prec_2\overline{1}\prec_21\prec_22$. It follows that $\alpha$ must send $3,5,1$ to each other and $4,6,2$ to each other, and thus is not a 6-cycle. 
\end{example}

\begin{proof}[Proof of Theorem \ref{thr:Grassmanianqthpower}]
    The cycle decomposition of a permutation $\pi$ for which $\pi^k=\text{id}$ only contains cycles whose lengths are divisors of $k$.
    Since $\pi$ is Grassmanian, every cycle has to be Grassmanian. Also, $\pi(1) \not=1$, $\pi(n) \not=n$ and $\des(\pi)=1$ implies that $\pi$ has no fixed point, or equivalently 1-cycle.

    By applying~\cref{lem:existence_composition} iteratively, for any solution to Equation~\eqref{eq:count_total_order}, we can combine $x_{d,i}$ Grassmanian $d$-cycles of type $i$ over all $d\in\mathcal{D}_k$ and $i\in[N_d]$ together into a single Grassmanian permutation whose $k$-th power is the identity. On the other hand, if such a collection of cycles can be combined in two different ways, then we can find two of these cycles that do not have the same relative order in these two combinations, contradicting~\cref{lem:determined_order}.  
    
    For a prime $p$, $N_p=\frac1p(2^p-2)$ by~\cref{lem:Np}, and Equation~\eqref{eq:count_total_order} reduces to $\sum_{i=1}^{N_p}x_i=\frac np$. If $p \nmid n$, there is no non-negative integer solution. If $p \mid n$, it is well-known that this equation has $\binom{N_p+\frac np -1}{N_p-1}$ solutions in non-negative integers.
\end{proof}

\bibliographystyle{abbrv}
\bibliography{bibliography}
\end{document}